\documentclass[a4paper,12pt]{article}
\usepackage{amsmath}
\usepackage{amsfonts}
\usepackage{eucal}
\usepackage{amsthm}
\usepackage{amssymb,amscd}
\usepackage{mathrsfs}
\numberwithin{equation}{section}
\usepackage{enumerate}

\newcommand{\norm}[1]{\| #1 \|}

\newcommand{\abs}[1]{| #1 |}

\newcommand{\Bigabs}[1]{\Bigl| #1 \Bigr|}

\newcommand{\jap}[1]{\langle #1 \rangle}

\def\a{\alpha}
\def\b{\beta}

\def\d{\delta}
\def\e{\varepsilon}
\def\f{\varphi}

\def\k{\kappa}

\def\m{\mu}

\def\t{\tau}
\def\x{\xi}
\def\y{\eta}

\def\th{\theta}

\renewcommand{\O}{\Omega}

\newcommand{\OP}{\mathrm{Op}}

\def\re{\mathbb{R}}

\def\na{\mathbb{N}}
\def\pa{\partial}

\newtheorem{thm}{Theorem}[section]
\newtheorem{lem}[thm]{Lemma}
\newtheorem{prop}[thm]{Proposition}
\newtheorem{cor}[thm]{Corollary}

\theoremstyle{definition}
\newtheorem{defn}{Definition}[section]
\newtheorem{ass}{Assumption}

\theoremstyle{remark}

\newtheorem*{thm*}{Theorem}
\newtheorem*{preack}{Acknowledgment}
\newenvironment{ack}{\begin{preack}\normalfont}{\end{preack}}

\theoremstyle{remark}
\newtheorem*{rem*}{Remark}

	\title{Semiclassical Defect Measures and the Observability Estimate for Schr\"odinger Operators with Homogeneous Potentials of Order Zero}
	\author{
		Keita M{\sc ikami}%
		\footnote{Graduate School of Mathematical Sciences, University of Tokyo, 3-8-1 Komaba, Meguro Tokyo, 153-8914 Japan. E-mail: {\tt kmikami@ms.u-tokyo.ac.jp}.  }
		}

\begin{document}

		\maketitle
		\begin{abstract}
			We study the asymptotic behavior as $\abs{x}\to\infty$ of Schr\"odinger operators with homogeneous potentials. For this purpose, we use methods from semiclassical analysis and investigate semiclassical defect measures. We prove their localization in the direction we apply to obtain a necessary observability condition.
		\end{abstract}



		\section{Introduction}

		Let $P=-\bigtriangleup + V$ be a Schr\"odinger operator on $\re^n$.
		We make the following assumption on the potential $V$.
		\begin{ass}
			(1) $V$ is a real valued smooth function.\\
			(2) We can decompose $V$ as $V=V_{\infty}+V_s$; here $V_{\infty}$ is real-valued, smooth, and homogeneous of order zero, i.e., $V_{\infty}(x)=V_{\infty}(\frac{x}{\abs{x}})$ for $\abs{x}\geq1$, and $V_s(x)=o(x^{-1})$ as $\abs{x}\to\infty$.
		\end{ass}

		If $V_{\infty}$ is homogeneous of order zero, one can regard $V_{\infty}$ as a function on $S^{n-1}$. We use the same symbol $V_{\infty}$ for the original potential and restriction of this function to $S^{n-1}$.

		Let $j \in C^{\infty}(\re:[0,1])$ be such that $j(r)=0$ if $r \leq \frac12$ and $j(r)=1$ if $1 \leq r$. For $a \in C^{\infty}_0(\re \times T^*S^{n-1})$, one can regard $j(r)a(\rho,\th,\frac{\y}{r})$ as an element of $C^{\infty}(T^*\re^n)$, here we use polar coordinate $(r,\th)$ of $\re^n\setminus\{0\}$ and its dual $(\rho,\y)$.
		The function $\tilde a(x,\x)=j(r)a(\rho,\th,\frac{\y}{r})$ on $T^*\re^n$ is in the symbol class $S$. Then Weyl quantization $\tilde a^{\mathrm{w}}(hx,D_x)$ of the symbol $\tilde a$ becomes a well-defined bounded linear operator on $L^2(\re^n)$. We write  $\tilde a^{\mathrm{w}}(hx,D_X)=\OP_{j}(a)$ (See Section 2.1 for the detail).

		\begin{thm}
		 (Existence of semiclassical defect measure for $\mathrm{Op}_j$)\\
		 Let $u_h\in L^2(\re^n)$ be a bounded sequence in $h$. There exists a sequence of positive numbers $h_m$  and a positive finite Radon measure $\m_{j}$ on $\re \times T^*S^{n-1}$ such that $h_m \to 0$ as $m \to \infty$ and
		 \begin{align*}
		 \jap{u_{h_m},\OP_{j}(a)u_{h_m}}_{L^2(\re^n)} \to \int_{\re \times T^*S^{n-1}} a \mathrm{d}\m_{j} \text{ as } m \to \infty,
		 \end{align*}
		 for all $a \in C^{\infty}_0(\re \times T^*S^{n-1})$.
	 \end{thm}

	 Wigner first introduced the notion of the semiclassical measure in \cite{wig}. The study of the partial differential equation using defect measure appeared in \cite{pl}, and Patrick-G\`erard refined it in \cite{pg}. You can find several proofs of the existence of semiclassical measures in \cite{cdv, ger, lp, zel}.
	Readers can find a good survey of this subject in \cite{am}.

	Usually, we define a semiclassical defect measure as a measure on a cotangent bundle. Roughly speaking, this standard semiclassical defect measure treats a point in the cotangent bundle whose orbits of the Hamiltonian flow generated by $p$ are trapped. One can prove that if the Hamiltonian flow generated by $p$ is non-trapping, $\m$ is identically zero. Under some assumptions, Schr\"odinger operators with homogeneous potentials are non-trapping(See Section 2 of \cite{h} for the detail). Thus we cannot apply the usual semiclassical analysis.

	Our idea is to consider a point in the cotangent bundle whose orbits of the Hamiltonian flow generated by $p$ scatters. We realize this idea by taking the position to infinity instead of taking the energy to infinity.
	One can find a non-semiclassical quantization similar to our new quantization in \cite{chs}.

	\begin{ass}
		Let $u_h\in\mathcal{D}(P)$ be such that
		\begin{eqnarray}
		\left\{
		\begin{array}{l}
		(P-E)u_h=R_h\\
		\norm{u_h}_{L^2(\re^n)}=1,
		\end{array}
		\right.
		\end{eqnarray}
		\begin{enumerate}[(\text{B}-1)]
			\item There exist $c_{\ell}:(0,1)\to(0,\infty)$, $\ell=1,2$ such that $c_2$ is monotone increasing, $0\leq c_1 \leq c_2$ and $c_2\to0$ as $h\to0$.
			\item Let $\tilde j_{h}(x) = j(c_2(h)\abs{x}) - j(c_1(h)\abs{x})$. $u_h=\tilde j_h u_h+R_h^{'}$ with $\norm{R_h^{'}}_{L^2(\re^n)}=o(1)$ as $h\to0$.
			\item There exists $h_0>0$ and $\d\in (1,2)$ such that $c_2(h)^{\d} < c_1(h)$ for each $h\in(0,h_0)$.
		\end{enumerate}
	\end{ass}
	If $c_2$ satisfies Assumption (B-1), then $c_2$ is injective. Since $c_2(h)\to0$ as $h\to0$ under Assumption (B-1), $c_2^{-1}:(0,\e_0)\to(0,1)$ exists for sufficiently small $\e_0>0$. Let $\tilde h=c_2(h)$.
	One can easily see that $h\to0$ as $\tilde h \to0$. Let $v_{\tilde h}=u_{c_2^{-1}(\tilde h)}$

		\begin{thm}
			Suppose Assumption (A) and (B) are satisfied and $\m_{j}$ be a defect measure of $v_{\tilde h}$. If $R_h =o(c_1(h))$ as $h\to0$, following statements hold:\\
				(1) $E\in \mathrm{Cr}(V)$,\\
				(2)  $\mathrm{supp} (\m_{j}) \subset \{(0,\th,0)\in\re\times T^*S^{n-1} \mid \th \in \mathrm{Cr}(V)\cap V^{-1}(E)\}$.
		\end{thm}

		Following corollary is a direct consequence of Theorem 1.2
		\begin{cor}
			Suppose Assumption (A) and (B) are satisfied with $c_1(h)=C_1h$ and $c_2(h)=C_2h$ for some $C_1,C_2>0$. If $R_h =o(h)$ as $h\to0$, following statements hold:\\
			(1) $E\in \mathrm{Cr}(V)$,\\
			(2) $\mathrm{supp} (\m_{j}) \subset \{(0,\th,0)\in\re\times T^*S^{n-1} \mid \th \in \mathrm{Cr}(V)\cap V^{-1}(E)\}$.
		\end{cor}

		The statement of (1) in Theorem 1.2 implies that there are not so many $o(c_1(h))$-quasimodes whose support escapes from the origin with $c_2(h)^{-1}$ order.

		When $c_1(h)=C_1h$ and $c_2(h)=C_2h$ with $C_1<C_2$, we can construct $o(h)$-quasimodes whose support escapes from the origin with $h^{-2}$ order (See Section 4).

		 The claim of (2) in Theorem 1.2 is a semiclassical analog of the localization of the solution in direction proved in \cite{hmv,hmv2,h,hs2}.
		 Let $\th\in S^{n-1}$. We define $\mathcal{H}_{\th} \subset L^2(\re^n)$ by $\{\f\in L^2(\re^n) \mid (\frac{x}{\abs{x}}-\th)e^{-itP}\f\rightarrow 0 \text{ as t} \rightarrow \infty\}$.
		 Then it is shown that there exists $\{\th_m\}^{M}_{m=1}$ such that $L^2(\re^n) = \oplus^M_{m=1} \mathcal{H}_{\th_m}$. When the potential is Morse function on $S^{n-1}$, $\{\th_m\}^{M}_{m=1}$ are local minimum of $V_{\infty}$.
		 This localization of the solution in direction extends to the case of scattering manifold by Hassell, Melrose, and Vasy in \cite{hmv} and \cite{hmv2} under the setting of \cite{mel}.

			One difference of the localization of the solution in the direction in \cite{hmv,hmv2,h,hs2} from our version is the appearance of $L^2$ states, which localize to saddle point and the local minimum points.
			In \cite{hmv}, Hassell-Melrose-Vasy showed that there is a distribution that localizes in saddle points or local minimum points.
			The appearance of $L^2$ states that localizes to saddle point and the local minimum points is essential. This appearance is because one can take $u_h$ such that $\m_j\neq 0$ and the support of $\m_j$ is in the direction of local maxima or saddle point (See Section 4).

			One can also show a result on the relationship between quasimodes and their support.
			\begin{ass}
				Let $u_h\in\mathcal{D}(P)$ satisfy (1.1).
				\begin{enumerate}[(\text{C}-1)]
					\item There exist $c_{\ell}:(0,1)\to(0,\infty)$, $\ell=1,2$ such that $c_2$ is monotone increasing, $0\leq c_1 \leq c_2$ and $c_2\to0$ as $h\to0$.
					\item Let $j_{\e}(x)=j(\e\abs{x})$ for $\e>0$. $u_{h}= j_{c_2(h)} u_{h}+R_{h}^{'}$ with $\norm{R_{ h}^{'}}_{L^2(\re^n)}=o(1)$ as $h\to0$.
					\item There exists $h_0>0$ and $\d\in (1,2)$ such that $c_2(h)^{\d} < c_1(h)$ for each $h\in(0,h_0)$.
				\end{enumerate}
			\end{ass}

		\begin{thm}
			Suppose Assumption (A) and (C) are satisfied. Further, we assume  $\norm{R_h}_{L^2(\re^n)}=o(c_1(h))$ as $h\to0$. Then either\\
			(1) $E\in \mathrm{Cr}(V)$\\
			(2) $(1- j_{c_2(h)})u_{h}\to 0$ as $h \to 0$
			implies $(1- j_{c_1(h)})u_{h}\to 0$
			 as $h \to 0$.\\
			holds.
		\end{thm}

		\begin{rem*}
			Theorem 1.4 states that if $E\notin \mathrm{Cr}(V)$, quasimodes supported in $\{\abs{x}>c_2(h)^{-1}\}$ must supported in $\{x\in\re^n\mid \abs{x}>\frac12c(h)^{-1}\}$.
			This implies that any $o(h)$-quasimodes that escapes from the origin faster than $h^{-1}$ order has a certain minimum escaping order faster than $h^{-1}$ if $E\notin \mathrm{Cr}(V)$.
		\end{rem*}

		We turn to the application of our semiclassical measure. We can prove an observability result. Let $\O\subset\re^n$ be a measurable set. We say observability holds on $\O$ if for some $T>0$ there exists $C_{\O,T}>0$ such that
		\begin{align*}
		& \norm{u}_{L^2(\re^n)}\leq C_{\O,T}\int^T_0\int_{\O}\abs{e^{-itP}u(x)}^2\mathrm{d}x\mathrm{d}t
		\end{align*}
		for any $u\in L^2(\re^n)$.

		\begin{thm}
			Let $\O\subset\re^n$ be a domain such that.

			$\O\cap\{x\in\re^n\mid\abs{x}>R\}\subset \re^n\setminus \{(r,\th)\in\re^n\mid r>R,\mathrm{dist}(\th,\th_0)<Cr^{-\ell}\}$ for some $R,C>0$, $0<\ell<1$ and $\th_0\in S^{n-1}$.

			Then the observability estimate on $\O$ fails for any $T>0$, i.e., there exists $u_m\in L^2(\re^n)$ such that $\norm{u_m}_{L^2(\re^n)}=1$ and $\int^T_0\int_{\O}\abs{e^{-itP}u_m(x)}^2\mathrm{d}x\mathrm{d}t\to0$ as $m\to\infty$.
		\end{thm}

		It is known that observability is equivalent to the controllability in \cite{li}. The controllability means the condition that for any $u_0\in L^2(\re^n)$ there exists $f\in L^2((0,T)\times\O)$ such that the solution to the equation
		\begin{eqnarray*}
			\left\{
			\begin{array}{l}
				(i\pa_t+P)u(t,x)=f\chi_{(0,T)\times\O}(t,x)\\
				u(0,x)=u_0(x).
			\end{array}
			\right.
		\end{eqnarray*}
		satisfies $u(t,x)\equiv0$.

		In \cite{lb}, Lebeau showed the observability of if the manifold is compact and the geodesic satisfies geometric control condition. Lebeau uses a semiclassical defect measure in that proof.

		The plot of this paper is as follows.
		We first introduce a new semiclassical quantization and give some basic properties to prove Theorem 1.1 in section 2. We also prove some results in classical mechanics in Section 2.
		In section 3, we prove Theorem 1.2 and 1.4. The proof is similar to that of the Hamiltonian flow invariance of the usual semiclassical defect measures.
		We construct an example of $u_h$ whose semiclassical defect measure $\m$ is not identically zero in Section 4. Finally, we give proof of Theorem 1.5 in Section 5.

			\begin{ack}
				The author is grateful to Professor Fabricio Maci\`a and Professor Shu Nakamura for suggesting the idea of the new semiclassical quantization in personal communication.
				The author is also grateful to Professor Erik Skibsted for introducing the author to an unpublished preprint and having many discussions. The author is also thankful to Professor Kenichi Ito and Kouichi Taira for helpful discussion. The author is also grateful to Genki Sato for correcting many grammatical mistakes.
				The author is under the support of the FMSP(Frontiers of Mathematics Science and Physics) program at the Graduate School of Mathematical Sciences, University of Tokyo. Also, he is supported by JSPS KAKENHI Grant Number 18J12370.
			\end{ack}

		\section{Preliminaries}
		\subsection{Admissible cut-off function}

		This subsection introduces the notion of the admissible cut-off function and proves the generalized version of Theorem 1.1 (see Theorem 2.1 for the detail).

		\begin{defn}(admissible cutoff function)\\
			Let $\{f_h\}_{h\in(0,1)} \subset C^{\infty}(\re)$ be a family of functions. We say $f_h$ is an admissible cutoff function if $f_h$ satisfies the following conditions:\\
			(1) $f_h(r)=0$ if $r \leq \e$ for some $\e>0$ independent in $h$.\\
			(2) There exists $\d\in(0,\frac12)$ such that, for any $m\in\na$, there exists $C_m>0$ such that $\sup_{r\in\re}{\abs{\pa_r^mf_h(r)}} <C_mh^{-\d m}$ uniformly in $h$.
		\end{defn}

		Let $a \in C^{\infty}(\re \times T^*S^{n-1})$ and admissible cutoff function $f_h$. One can regard $f_h(r)a(\rho,\th,\frac{\y}{r})$ as an element of $C^{\infty}(T^*\re^n)$ for small $h$ with the natural diffeomorphism $T^*\re_{> 0 (r,\rho)} \times T^*S^{n-1}_{(\th,\y)} \simeq T^*\re^n\setminus\{0\}_{(x,\x)}$ induced by the polar coordinate.
		Then the function $\tilde a_h(x,\x)=f_h(r)a(\rho,\th,\frac{\y}{r})$ on $T^*\re^n$ is well-defined. If the all derivatives of the $a$ is uniformly bounded on $\re\times T^{*}S^{n-1}$, $\tilde a$ is in the symbol class $S$, where
		\begin{align*}
			S=\{a\in C^{\infty}(T^{*}\re^n) \mid \forall \a,\b\in \na^n, \sup_{(x,\x)\in T^{*}\re^n}\abs{\pa_x^{\a}\pa_{\x}^{\b} a(x,\x)}< \infty\}.
		\end{align*}
		For the readers' convince, we will prove $\tilde a(x,\x) \in S$.

		Let $v\in T_{x}\left(\re^n\setminus\{0\}\right)$. From the cartesian coordinates, we can write $v=\sum_{m=1}^nv_i\pa_{x_i}$. Also, if we fix local coordinate $(U,\psi)$ of $S^{n-1}$ with $\frac{x}{\abs{x}}\in U$, we can write $v=v_r\pa_r+\sum_{m=1}^{n-1}v_{\th_i}\pa_{\th_i}$ using polar coordinate.

		Let $\tilde \psi$ be a map $(0,\infty)\times U \subset \re^n \to \psi(U)$ which takes $y$ to $\psi(\frac{y}{\abs{y}})$. Then we can write $v_r=\frac{x}{\abs{x}}\cdot \overrightarrow{v}$ and $v_{\th}=J(\tilde\psi) \frac1{\abs{x}}\{I_n-(\frac{x_ix_j}{\abs{x}^2})_{i,j}\}\overrightarrow{v}$ where $\overrightarrow{v}=(v_1,v_2,\cdots,v_n)$ and $J(\tilde\psi)$ denotes the Jacobi matrix of $\tilde\psi$ at $x$.

		Let $\x\in T^{*}\re^n$. Then using dual coordinate of cartesian coordinates and polar coordinate, $\x$ can be written as $\sum_{m=1}^n \x_i\mathrm{d}x_i$ and $\rho\mathrm{d}r+\sum_{m=1}^{n-1}\y_i\mathrm{d}\th_i$.

		Using cartesian coordinate, we see $\x(v)=\sum_{m=1}^n \x_iv_i$. Using polar coordinate, we see $\x(v)=\rho\frac{x}{\abs{x}}\cdot \overrightarrow{v}+\y\cdot J(\tilde\psi) \frac1{\abs{x}}\{I_n-(\frac{x_ix_j}{\abs{x}^2})_{i,j}\}\overrightarrow{v}$.

		Substituting $\overrightarrow{v}=\frac{x}{\abs{x}}$, we see $\rho=\frac{x}{\abs{x}}\cdot\x$. If $\{I_n-(\frac{x_ix_j}{\abs{x}^2})_{i,j}\}\overrightarrow{v}=\overrightarrow{v}$ i.e. $\frac{x}{\abs{x}}\cdot\overrightarrow{v}=0$, we see $\abs{x}\x\cdot\overrightarrow{v}=\y\cdot J(\tilde\psi)\overrightarrow{v}$, which means $\abs{x}\x={}^tJ(\tilde\psi)\y$.

		Thus we obtain
		\begin{align*}
		& \forall \a,\b\in \na^n, \sup_{(x,\x)\in T^{*}\re^n}\abs{\pa_x^{\a}\pa_{\x}^{\b}\tilde a(x,\x)}<\infty \\
		& \Leftrightarrow \forall m,\ell\in\na, \tilde\a,\tilde\b\in \na^{n-1}, \sup_{(x,\x)\in T^{*}\re^n}\abs{\pa_r^{m}\pa_{\rho}^{\ell}\pa_{\th}^{\tilde\a}(r\pa_{\y})^{\tilde\b}\tilde a(x,\x)}<\infty.
		\end{align*}
		Thus $\tilde a(x,\x) \in S$.

		 Now Weyl quantization $\tilde a^{\mathrm{w}}_h(hx,D_x)$ of the symbol $\tilde a$ becomes a well-defined bounded linear operator on $L^2(\re^n)$, given as the extension of
 		 \begin{align*}
 		 & \tilde a^w_h(hx,D_x)u(x)=\frac{1}{(2\pi)^n}\int_{\re^{2n}}e^{i(x-y)\cdot\x}\tilde a_h\left(\frac{h(x+y)}{2},\x\right)u(y)\mathrm{d}y\mathrm{d}\x,
 		 \end{align*}
 		 for $u\in \mathscr{S} (\re^n)$. We write  $\tilde a_{h}^{\mathrm{w}}(hx,D_X)=\OP_{f_h}(a)$.

		 \begin{thm}
			(Existence of semiclassical defect measure for the admissible cut-off function)\\
			Let $u_h\in L^2(\re^n)$ be a bounded sequence in $h$ and $f_h$ be a admissible cut-off function. There exists a sequence of positive numbers $h_m$ such that $h_m \to 0$ as $m \to \infty$ and a finite Radon measure $\m_{f}$ on $\re \times T^*S^{n-1}$ and
			\begin{align*}
			\jap{u_{h_m},\OP_{f_{h_m}}(a)u_{h_m}}_{L^2(\re^n)} \to \int_{\re \times T^*S^{n-1}} a \mathrm{d}\m_{f} \text{ as } m \to \infty,
			\end{align*}
			for all $a \in C^{\infty}_0(\re \times T^*S^{n-1})$. Furthermore, if $f_h$ is non negative, $\m_{f}$ is also non negative.
		\end{thm}

		\begin{rem*}
			$j$ in section 1 is appearently an admissible cut-off function. Thus Theorem 1.1 follows from Theorem 2.1.
		\end{rem*}

		 We define dilation operator $U_h$ by $U_hu(x)=h^{\frac{n}2}u(\frac{x}{h})$ for $u\in L^2(\re^n)$. Then $U_h$ is unitary and one can calculate
 		\begin{align}
 		& \OP_{f_h}(a)= U_h^{-1}\tilde a^{\mathrm{w}}(X,hD_X) U_h.
 		\end{align}
 		 Thus we can apply a usual semiclassical analysis for $S$.
 		 Then one can use results in usual semiclassical analysis in \cite{zws} to obtain the following theorems.

		 \begin{thm}(Calderon-Vaillancourt Theorem)\\
		 	For $a \in S$, there exists $C>0$ such that\\
		 	$\norm{a^{\mathrm{w}}(hx,D_X)}_{\mathcal{L}(L^2(\re^n))} \leq C\sup_{(x,\x)\in\re^{2n}}\abs{a(x,\x)} + \mathcal{O}(h^{\frac12})$ as $ h \to 0$.
		 \end{thm}

		\begin{thm}(Sharp G$\mathring{a}$rding inequality)\\
			Suppose $a \in C^{\infty}_0(\re \times T^*S^{n-1}:[0,\infty))$. Then there exist $C>0$ and $h_0 >0$ such that
			\begin{align*}
			& \jap{u,\OP_{f_h}(a)u}_{L^2(\re^n)} \geq -Ch\norm{u}^2_{L^2(\re^n)}
			\end{align*}
			for $u \in L^2(\re^n)$ and $0<h<h_0$.
		\end{thm}

		\begin{proof}[Proof of Theorem 2.1]
			The proof is essentially the same as that of Theorem 5.2 in \cite{zws}. However, we give the detail for completeness.

			Since $C_0(\re \times T^*S^{n-1})$ is separable with the topology defined by sup-norm and $C^{\infty}_0(\re \times T^*S^{n-1})$ is dense subspace, thus one can find $\{a_{\ell}\}\in C^{\infty}_0(\re \times T^*S^{n-1})$ which is dense in $C_0(\re \times T^*S^{n-1})$.

			From Theorem 2.2, $\jap{u_h,\OP_{f_h}(a_1)u_h}$ is bounded in $h$. Thus one can find sequence $h^{(1)}_m$ such that $h^{(1)}_m \to 0$ and $\jap{u_{h^{(1)}_m},\OP_{f_{h^{(1)}_m}}(a_1)u_{h^{(1)}_m}} \to F(a_1)$ as $m \to \infty$ for some $F(a_1)$.

			 Similarly, for $\ell=2,3,4,\cdots$ one can find sequence $h^{(\ell)}_m$ which is subsequence of $h^{(\ell-1)}_m$ and $\jap{u_{h^{(\ell)}_m},\OP_{f_{h^{(\ell)}_m}}(a_{\ell})u_{h^{(\ell)}_m}} \to F(a_{\ell})$ as $m \to \infty$ for some $F(a_{\ell})$. Then by diagonal argument one can find sequence $h_m$ such that $\jap{u_{h_m},\OP_{f_{h_m}}(a_l)u_{h_m}} \to F(a_{\ell})$ as $m \to \infty$ for each $\ell$.

			 From Theorem 2.2, one can calculate as follows:
			 \begin{align*}
				& \jap{u_{h_m},\OP_{f_{h_m}}(a_{\ell})u_{h_m}}\\
				& \leq \norm{\OP_{f_{h_m}}(a_{\ell})}_{\mathcal{L}(L^2(\re^n))} \norm{u_{h_m}}_{L^2(\re^n)}\\
				& \leq C\sup_{(r,\rho,\th,\y)\in\re^{2n}}\abs{f_{h_m}(r)a_{\ell}(\rho,\th,\y)} + \mathcal{O}(h^{\frac12})\\
				& \leq C\sup_{(\rho,\th,\y)\in\re \times T^*S^{n-1}}\abs{a_{\ell}(\rho,\th,\y)} + \mathcal{O}(h^{\frac12}),
			\end{align*}
			 where we have used the fact $f_h$ is uniformly bounded in the last line.
			 Thus a functional $a_{\ell} \mapsto F(a_{\ell})$ defines a bounded and linear functional $F$ on $C_0(\re \times T^*S^{n-1})$.
			  Then Riesz-Markov-Kakutani theorem implies there exists a Radon measure $\m_{f}$ such that $\jap{u_{h_m},\OP_{f_{h_m}}(a)u_{h_m}} \to \int_{\re \times T^*S^{n-1}}a \mathrm{d}\m_{f}$ as $m \to \infty$ for any $a\in C^{\infty}_0(\re \times T^*S^{n-1})$.

			 Taking $\chi_n \in C_0(\re \times T^*S^{n-1})$ such that $0\leq\chi_n \nearrow 1$ pointwise as $n \to \infty$. Then one obtains $\lim_{n\to \infty} \int_{\re \times T^*S^{n-1}}\chi_n \mathrm{d}\m_{f}= \m_{f}(\re \times T^*S^{n-1})$ from the monotone convergence theorem. Since $f_h$ is uniformly bounded, Theorem 2.2 implies $\lim_{n\to \infty} \int_{\re \times T^*S^{n-1}}\chi_n \mathrm{d}\m_{f} \leq C$. This means $\m_{f}(\re \times T^*S^{n-1}) \leq C$, which proves finiteness.

			 Theorem 2.3 implies that $F$ is non-negative if $f_h$ is non-negative. Thus $\m_f$ is positive provided $f_h$ is positive
		\end{proof}

		\begin{lem}
			Let $a\in C^{\infty}(\re\times T^{*}S^{n-1})$ and $f_h$ be an admissible cut-off function. If $a$ and all the derivaties of $a$ are uniformly bounded on $\re\times T^{*}S^{n-1}$, there exist $h_0$ and $C>0$ such that
			\begin{align*}
				& \norm{\mathrm{Op}_{f_h}(a)u}_{L^2(\re^n)}\leq C\left(\norm{f_hu}_{L^2(\re^n)} + h^{\frac12}\norm{u}_{L^2(\re^n)}\right)
			\end{align*}
			for any $u\in L^2(\re^n)$ and $h\in(0,h_0)$.
		\end{lem}
		\begin{proof}
			Let $\tilde a (x,\x)=f_h(r)a(\rho,\th,\frac{\y}{r})$, Then $\abs{\tilde a (x,\x)}^2 \leq C f_h^2$ since $a$ is uniformly bounded on $\re \times T^{*}S^{n-1}$. From Theorem 2.3, we obtain
			\begin{align*}
				& \norm{\mathrm{Op}_{f_h}(a)u}_{L^2(\re^n)}^2\\
				& = \jap{u, \left(\abs{\tilde a}^2\right)^{\mathrm{w}}(hx,D_x)u}_{L^2(\re^n)} +\mathcal{O}(h)\\
				& \leq C\norm{f_hu}_{L^2(\re^n)}^2 +Ch\norm{u}_{L^2(\re^n)}^2,
			\end{align*}
			which concludes the proof.
		\end{proof}

	\subsection{Induced Dynamical System}
	Here we consider the following dynamical system on $\re\times T^*S^{n-1}$, induced by a Hamiltonian flow of $P$. Herbst first proved the contents of this section in \cite{h}, but we write here for the convince.

	Let $H$ be a vector field on $T^*(\re\times T^*S^{n-1})$ defined by
	\begin{align*}
	&  H= q( \th,\y)\pa_{\rho}+(\pa_{\y}q)( \th,\y))\pa_{\th} -((\pa_{\th}q)( \th,\y))+ \{\pa_{\th}V_{\infty})(\th)+2\rho\y\}\pa_{\y},
	\end{align*}
	where $q(\th,\y)={}^t\y h(\th)\y$ is the symbol of Laplacian on $S^{n-1}$.

	The relation of this dynamical system and the Schr\"odinger operator with homogeneous potential is as follows:

	Let $\tilde \Phi_t$ be a Hamiltonian flow generated by the Hamiltonian of $H$.
	For $(r,\rho,\th,\y)\in T^*\re^n$ we write $\tilde \Phi_t(r,\rho,\th,\y)=(\tilde r(t),\tilde \rho(t),\tilde \th(t),\tilde \y(t))$.

	Then $(\tilde r(t),\tilde \rho(t),\tilde \th(t),\tilde \y(t))$ satisfy
	\begin{align*}
	& \frac{\mathrm{d}}{\mathrm{d}t}\tilde r(t) = 2\tilde \rho(t), \frac{\mathrm{d}}{\mathrm{d}t}\tilde \rho(t) = q( \tilde \th(t),\frac{\tilde \y(t)}{\tilde r(t)}),\\
	& \frac{\mathrm{d}}{\mathrm{d}t}\tilde \th(t) = (\pa_{\y}q)( \tilde \th(t),\frac{\tilde \y(t)}{\tilde r(t)}),
	 \frac{\mathrm{d}}{\mathrm{d}t}\tilde \y(t) =-\{(\pa_{\th}q)( \tilde \th(t),\frac{\tilde \y(t)}{\tilde r(t)})+(\pa_{\th}V_{\infty})(\tilde \th(t))\}.
	\end{align*}

	If we take $(\rho(t),\th(t),\y(t))=(\tilde \rho(t),\tilde \th(t),\frac{\tilde \y(t)}{\tilde r(t)})$, we obtain
	\begin{align*}
	& \frac{\mathrm{d}}{\mathrm{d}t}\rho(t) = \tilde r(t)^{-1}q( \th(t), \y(t)),
	\frac{\mathrm{d}}{\mathrm{d}t}\th(t) = \tilde r(t)^{-1}(\pa_{\y}q)( \th(t),\y(t)),\\
	& \frac{\mathrm{d}}{\mathrm{d}t} \y(t) =-\tilde r(t)^{-1}\{(\pa_{\th}q)( \th(t),\y(t)) + (\pa_{\th}V)(\th(t)) +2\rho(t)\y(t)\}.
	\end{align*}

	We assume $\tilde r(t) \neq 0$ and $\tilde r(t)\to\infty$ as $t \to \infty$. We introduce new time $\t$ by $\t=\int^t_0 \tilde r(s)^{-1} \mathrm{d}s$. Then we see
	\begin{align*}
	& \frac{\mathrm{d}}{\mathrm{d}\t}\rho(t) =q( \th(t), \y(t)),
	\frac{\mathrm{d}}{\mathrm{d}\t}\th(t) =(\pa_{\y}q)( \th(t),\y(t)),\\
	& \frac{\mathrm{d}}{\mathrm{d}\t} \y(t) =-\{(\pa_{\th}q)( \th(t),\y(t))+(\pa_{\th}V_{\infty})(\th(t)) +2\rho(t)\y(t)\}.
	\end{align*}
	Thus considering the orbit of $\Phi_t$ corresponds to viewing the orbit of Hamilton flow of $P$.

	In the last of this section, we write $\Phi_t(\rho,\th,\y)=(\rho(t),\th(t),\y(t))$ for $(\rho,\th,\y)\in\re \times T^*S^{n-1}$.

	\begin{lem}
		Total energy $\rho^2+q(\th,\y)+V_{\infty}(\th)$ is conserved.
	\end{lem}

	\begin{proof}
		Let $E(t)=\rho(t)^2+q(\th(t),\y(t))+V_{\infty}(\th(t))$. Then we see
		\begin{align*}
		& \frac{\mathrm{d}}{\mathrm{d}t}E(t)\\
		& =2\rho(t)q(\th(t),\y(t)) + \{(\pa_{\th}q)( \th(t),\y(t))+(\pa_{\th}V_{\infty})(\th(t))\}(\pa_{\y}q)( \th(t),\y(t))\\
		& - (\pa_{\y}q)( \th(t),\y(t))\{(\pa_{\th}q)( \th(t),\y(t))+(\pa_{\th}V_{\infty})(\th(t)) +2\rho(t)\y(t)\}\\
		& =0,
		\end{align*}
		which concludes the proof.
	\end{proof}

	\begin{lem}
		$\displaystyle\lim_{t\to\infty}\rho(t)$ and $\displaystyle\lim_{t\to-\infty}\rho(t)$ exist.
	\end{lem}
	\begin{rem*}
		Since $\frac{\mathrm{d}}{\mathrm{d}t}\rho(t)=q(\th(t),\y(t))$, $q(\th(t),\y(t))$ is integrable on $(0,\infty)$.
	\end{rem*}
	\begin{proof}
		We prove when $t\to\infty$. One can prove the assertion when $t\to-\infty$ similarly.

		Since $\frac{\mathrm{d}}{\mathrm{d}t}\rho(t)=q(\th(t),\y(t)) >0$, $\rho(t)$ is monotone increasing. From Lem2.4, $\rho(t)^2\leq \rho(t)^2 + q(\th(t),\y(t)) = E(0)-V_{\infty}(\th(t))$. Since $V_{\infty}$ is bounded, so is $\rho(t)$. Thus $\rho(t)$ is monotone increasing and bounded and has a limit.
	\end{proof}

	\begin{lem}
			$q(\th(t),(\pa_{\th}V_{\infty})(\th(t)))$ is integrable on $(0,\infty)$ with respect to $t$.
	\end{lem}
	\begin{proof}
		Let $F(t)=-{}^t(\pa_{\th}V_{\infty}(\th(t)))h(\th(t))\y(t)$. Then we obtain
		\begin{align*}
		& \frac{\mathrm{d}}{\mathrm{d}t}F(t)\\
		& = -{}^t(\mathrm{Hess}(V_{\infty})(\th(t))(\pa_{\y}q)( \th(t),\y(t)))h(\th(t))\y(t))\\
		& - {}^t(\pa_{\th}V_{\infty}(\th(t)))\{(\pa_{\th}h(\th(t)))(\pa_{\y}q)( \th(t),\y(t))\}\y(t)\\
		& + {}^t(\pa_{\th}V_{\infty}(\th(t)))h(\th(t))\{(\pa_{\th}q)( \th(t),\y(t))+(\pa_{\th}V_{\infty})(\th(t)) +2\rho(t)\y(t)\}
		\end{align*}
		Thus there exists $C>0$ such that
		\begin{align*}
		& \frac{\mathrm{d}}{\mathrm{d}t}F(t) +Cq(\th(t),\y(t)) > Cq(\th(t),(\pa_{\th}V_{\infty})(\th(t))).
		\end{align*}
		By integrating this inequality from $t=0$ to $t=T$, we obtain
		\begin{align*}
		& F(T)-F(0) +C\int^{T}_0 q(\th(t),\y(t))\mathrm{d}t > C\int^{T}_0q(\th(t),(\pa_{\th}V)(\th(t)))\mathrm{d}t.
		\end{align*}

		Since $q(\th(t),(\pa_{\th}V)(\th(t)))\geq0$ it is sufficient to show there exists a sequence $T_j$ such that $T_j\to\infty$ as $j\to\infty$ and $\{F(T_j)\}$ has upper bound.

		From the definition of $q$, we obtain
		\begin{align*}
		& \abs{F(t)}\leq q(\th(t),\y(t)) +q(\th(t),(\pa_{\th}V_{\infty})(\th(t)) .
		\end{align*}
		Since the second term is bounded, we only have to show that the first term is bounded for some $\{T_j\}$.
		Since first term is integrable, there exists a sequence $\{T_j\}$ such that there exist $C>0$ such that $q(\th(t),\y(t))<C$, which completes the proof.

	\end{proof}

	\begin{thm}
		$\lim_{t\to\infty}(\pa_{\th}V_{\infty})(\th(t)) = \lim_{t\to\infty}\y(t)= 0$.
	\end{thm}

	\begin{proof}
		Let $G(t)=q\left(\th(t),(\pa_{\th}V_{\infty})(\th(t))\right)$. Then we obtain
		\begin{align*}
		& \frac{\mathrm{d}}{\mathrm{d}t}G(t)\\
		& = 2{}^t(\mathrm{Hess}(V_{\infty})(\th(t))(\pa_{\y}q)( \th(t),\y(t)))h(\th(t))(\pa_{\th}V_{\infty})(\th(t))\\
		& + {}^t(\pa_{\th}V_{\infty}(\th(t)))\{(\pa_{\th}h(\th(t)))(\pa_{\y}q)( \th(t),\y(t))\}(\pa_{\th}V_{\infty})(\th(t)).
		\end{align*}
		Similar to the proof of Lemma 2.7, one can prove that the right-hand side is integrable, and $\lim_{t\to\infty}G(t)$ exists. Since $G(t)$ is integrable, this limit should be zero.

		Since
		\begin{align*}
		& \frac{\mathrm{d}}{\mathrm{d}t}V_{\infty}(\th(t))\\
		& = {}^t(\pa_{\th}V_{\infty}(\th(t)))h(\th(t))(\pa_{\y}q)( \th(t),\y(t))\\
		& \leq q(\th(t),(\pa_{\th}V_{\infty})(\th(t)) + q(\th(t),\y(t))
		\end{align*}
		$\lim_{t\to\infty}V_{\infty}(\th(t))$ exists. From Lem 2.5, $q(\th,\y(t))=E-\rho^2-V(\th(t))$ for some constant $E$. Since right hand side has limit as $t\to\infty$, $\lim_{t\to\infty}q(\th,\y(t))$ exists. Then integrability of $q(\th,\y(t))$ yields this limit is zero.
	\end{proof}

		One can prove simlar statement when $t\to\infty$ in a similar way.
		\begin{cor}
			$\lim_{t\to-\infty}(\pa_{\th}V_{\infty})(\th(t)) = \lim_{t\to-\infty}\y(t)= 0$.
	\end{cor}

	\section{Proof of Theorem 1.2 and Theorem 1.4}

	We first prepare a lemma and Theorem to prove Theorem 1.2.

	\begin{lem}
		Let $a \in C^{\infty}_0(\re \times T^*S^{n-1})$, then one obtains the following:
		\begin{align*}
		&[\OP_{f_h}(a),P]=\frac{h}{i} \{f_h(r)a(\rho,\th,\frac{\y}{r}),\rho^2+q(\th,\frac{\y}{r})+V_{\infty}(\th)\}^{\mathrm{w}}(hX,D_X)+ E_h
		\end{align*}
		as $h \to 0$, where $E_h$ is a family of pseudodifferential operator on $L^2(\re^n)$ depending on $h$ such that $\norm{E}_{\mathcal{L}(L^2(\re^n))}=o(h)$ as $h \to 0$. We note that $\{\cdot,\cdot\}$ denotes Poisson bracket.
	\end{lem}

	\begin{proof}
		From equality (2.1), one can directly obtain the assertion for $-\bigtriangleup$ from Theorem 4.18 in \cite{zws} i.e.
		\begin{align*}
		&[\OP_{f_h}(a),-\bigtriangleup]=\frac{h}{i} \{f_h(r)a(\rho,\th,\frac{\y}{r}),\rho^2+q(\th,\frac{\y}{r})\}^{\mathrm{w}}(hX,D_X)+ \mathcal{O}(h^3).
		\end{align*}

		Let $\k\in C^{\infty}(\re)$ be such that $\k(x)=1$ if $x>\e$ and $\k(x)=0$ if $x<\frac{\e}2$, where $\e>0$ is taken so that $f_h(x)=0$ if $x<\e$. Then we can calculate as follows:
		\begin{align*}
		&[\OP_{f_h}(a),V]=[\OP_{h,c}(a),\k(hr)(V_{\infty}+V_s)] + [\OP_{f_h}(a),\{1-\k(hr)\}V].
		\end{align*}

		Since $V_{\infty}$ is homogeneous of order zero, $k(h\abs{x})V_{\infty}(x)=\tilde V(hx)$ is a smooth and bounded function on $C^{\infty}(\re^n)$.
		Then one can obtain equality similarly to the case of $-\bigtriangleup$ from (2.1).

		Concerning $V_s$, one can calculate $\norm{k(hr)V_s}_{\mathcal{L}(L^2(\re^n))}=o(h)$ as $h\to0$ from the definition of $V_s$. Thus $[\OP_{f_h}(a),j(2c(h)hr)V_s]=o(h)$ as $h\to0$ from Theorem 2.2.

		Next we claim that $\OP_{f_h}(a)\{1-\k(hr)\}=\mathcal{O}(h^3)$.  Let $\tilde \k(x)=1-\k(\abs{x})$, then $\tilde \k \in C^{\infty}_0(\re^n)$ from the definition of $k$. By conjugating semiclassical dilation $U_h$, one can calculate as follows:
		\begin{align*}
		&\OP_{f_h}(a)\tilde \k(hx)\\
		& = \{f_h(r)a(\rho,\th,\frac{\y}{r})\}^{\mathrm{w}}(hx,D_x)\tilde \k(hx)\\
		& = U_h^* \{f_h(r)a(\rho,\th,\frac{\y}{r})\}^{\mathrm{w}}(x,hD_x) \tilde \k(x)U_h.
		\end{align*}

		Since $\mathrm{supp}(f_h(r)a(\rho,\th,\frac{\y}{r}))\cap\mathrm{supp}(\tilde \k(x)) = \phi$, Theorem 4.18 in \cite{zws} implies the claim.

		Since multiplication operator by $V$ is uniformly bounded in $h$, the claim implies $[\OP_{f_h}(a),\{1-\k(hr))\}V]=\mathcal{O}(h^3)$, which concludes the proof.

	\end{proof}

	\begin{thm}(Energy conservation)\\
		Let $f_h$ be an admissible cutoff function and let $u_h\in\mathcal{D}(P)$ be such that
		\begin{eqnarray*}
		\left\{
		\begin{array}{l}
		(P-E)u_h=R_h\\
		\norm{u_h}_{L^2(\re^n)}=1,
		\end{array}
		\right.
		\end{eqnarray*}
		where $\norm{R_h}_{L^2(\re^n)}=o(1)$ as $h\to0$. Under Assumption A, the support of $\m_{f}$ is localized in energy surfaces in the following meaning:
		\begin{align*}
		& \mathrm{supp}(\m_{f}) \subset
		\{(\rho,\th,\y)\in \re \times T^*S^{n-1} \mid \rho^2 + q(\th,\y) + V_{\infty}(\th)=E\}.
		\end{align*}

	\end{thm}

	\begin{proof}
		Since $(P-E)u_h=o(1)$, one can calculate
		\begin{align*}
		&o(1)=\jap{u_h,\OP_{f_h}(a)(P-E)u_h}_{L^2(\re^n)}\\
		& =\jap{u_h,\{f_h(r)a(\rho,\th,\y)(\rho^2+q(\th,\frac{\y}{r})+V_{\infty}(\th)-E)\}^{\mathrm{w}}(hX,D_X)u_h}_{L^2(\re^n)} + o(1)
		\end{align*}
		as $h\to0$ where we have used the fact that $\norm{\OP_{f_h}(a)V_s}_{\mathcal{L}(L^2(\re^n))}=o(1)$ as $h\to0$.

		Therefore, if we take a suitable subsequence $h_m$ and $m\to0$, we obtain $\int_{\re \times T^*S^{n-1}}a(\rho^2+q+V-E) \mathrm{d}\m_{f} =0$, which concludes the proof.
		\end{proof}

	 	\begin{lem}
	 		Under the same setting in Theorem 3.2, if $(1-f_h)u_h\to0$ as $h\to 0$, $\m_f>0$.
	 	\end{lem}

	 	\begin{proof}
	 		We prove by contradiciton. Assume $\m_f=0$.

	 		Since $\mathfrak{E}=\{(\rho,\th,\y)\in \re \times T^*S^{n-1} \mid \rho^2 + q(\th,\y) + V_{\infty}(\th)=E\}$ is compact, we can find $a\in C^{\infty}_0(\re\times T^{*}S^{n-1}:[0,1])$ such that $a=1$ on $\mathfrak{E}$.

	 		Then
	 		\begin{align*}
	 			& f_h-\mathrm{Op}_{f_h}(a) = -\mathrm{Op}_{f_h}(1-a) + E_h,
	 		\end{align*}
	 		where $\norm{E_h}=_{\mathcal{L}(L^2(\re^n))}=\mathcal{O}(h)$ as $h\to0$.

	 		Since $\abs{\rho^2 + q(\th,\y) + V_{\infty}(\th)-E}>\e$ for some $\e>0$ on supp$\left(1 - a \right)$,
			$b(\rho ,\th ,\y)=\left(1 - a(\rho,\th,\y)\right)(\rho^2 + q(\th,\y) +V_{\infty}(\th)-E)^{-1}$ and all  the derivatives of $b$ are uniformly bounded on $\re\times T^{*} S^{n-1}$.

	 		From  Assumption A, Theorem 2.2 and Lemma 2.4, we obtain
	 		\begin{align*}
	 			& \norm{\left(f_h-\mathrm{Op}_{f_h}(a)\right)u_h}_{L^2(\re^n)} \\
	 			&\leq  \norm{\mathrm{Op}_{f_h}(b)(-\bigtriangleup +V_{\infty}-E)u_h}_{L^2(\re^n)} + \mathcal{O}(h)\\
				& \leq  C\norm{f_h(-\bigtriangleup +V_{\infty}-E)u_h}_{L^2(\re^n)} + \mathcal{O}(h^{\frac12})\\
				& \leq \tilde C\norm{f_h(-\bigtriangleup +V-E)u_h}_{L^2(\re^n)}+ \mathcal{O}(h^{\frac12}) =\tilde C\norm{R_h}_{L^2(\re^n)}+ \mathcal{O}(h^{\frac12})
	 		\end{align*}
	 		as $h\to0$, where $C$ and $\tilde C$ are constants independent in $h$. Thus
			\begin{align*}
				& \left(f_h-\mathrm{Op}_{f_h}(a)\right)u_h\to 0
			\end{align*}
			as $h\to0$.

			Since $\m_f=0$,
	 		\begin{align*}
	 			& \displaystyle\lim_{h\to0}\norm{\mathrm{Op}_{f_h}(a)u_h}^2_{L^2(\re^n)} = \displaystyle\lim_{h\to0}\jap{u_h, \mathrm{Op}_{f_h}(a^2)u_h}^2_{L^2(\re^n)}=\int_{\re\times T^{*}S^{n-1}}a^2\mathrm{d}\m_{f}=0.
	 		\end{align*}
	 		Thus $\displaystyle\lim_{h\to0}\norm{u_h}=\displaystyle\lim_{h\to0}\norm{f_hu_h}=0$, which is a contradiction and we conclude the assertion.
	 	\end{proof}

		Recall $\tilde h =c_2(h)\to0$ and $h\to0$ are equivalent. Let  $j$ be the same with that in Section 1. We denote $j_{\e}(x)=j(\e x)$.

		Since $\tilde h=c_2(h)$ is monotone increasing, function in $h$ can be regarded as a function in $\tilde h$.
	 Let $c_3(\tilde h)=\max\{\norm{R_h}^{\frac12}c_1(h)^{\frac12}, c_1(h)^{\frac{2}{\d}}\}$, $c_4(\tilde h) = \frac{c_1(h)^2}{c_3(\tilde h)\tilde h}$ and $c_5(\tilde h)=\left(\frac{c_3(\tilde h)}{c_1(h)}\right)^2$.

	 We define
	 \begin{align*}
		 & \chi_{\tilde h}(x)= \left\{ \begin{array}{l}
		 j(2x)\left(1-j(-\frac{\log (c_4(\tilde h)x)}{\log c_5(\tilde h)})\right) \quad 0<x\\
		 0 \qquad \qquad \qquad \qquad \qquad \qquad x\leq0
		\end{array}\right.
	 \end{align*}
	 for $x\in\re$.

	 \begin{lem}
		 (1) If $x\in \mathrm{supp} \chi_{\tilde h}$, $x<\frac{\tilde h}{c_3(\tilde h)}$. \\
		 (2) $\frac12<x<\frac{\tilde h}{c_1(\tilde h)}$ implies $\chi_{\tilde h}(x)=1$.\\
		 (3) $ \chi_{\tilde h} (\tilde hx)\tilde j_{h}(x) =\tilde j_{h} (x)$.

	 \end{lem}
	 \begin{proof}
		  $x\in \mathrm{supp} \chi_{\tilde h}$ implies $\log{ (c_4(\tilde h)x) } < - \log c_5(\tilde h) $ by the definition of $\chi_{\tilde h}$. This implies $x< \frac1{c_4(\tilde h)c_5(\tilde h)} $.
		 Since $c_4(\tilde h)c_5(\tilde h)=\frac{c_3(\tilde h)}{\tilde h} $, the assertion of (1) follows.
		 The assertion of (2) follows similarly.

	 If $x\in\mathrm{supp}\tilde j_h$, $\frac12 \tilde h ^{-1}< x < c_1(h)^{-1}$. From (1) in this Lemma, we obtain $\chi_{\tilde h}( \tilde h x)=1$ if $x\in\mathrm{supp}\tilde j_h$, which concludes the proof.

	 \end{proof}

	 \begin{prop}
		 $ \chi_{\tilde h}$ is an admissible cut-off function.
	 \end{prop}

	 \begin{proof}
			From the definition of $j$ and $c_4,c_5$, for any $m\in\na\setminus\{1\}$, $j^{(m)}(2x)\neq 0$ implies $x\in(\frac12,1)$ and $j^{(m)}\left(-\frac{\log (c_4(\tilde h)x)}{\log c_5(\tilde h)} \right)\neq 0$ implies $x\in\left(\frac{\tilde h}{c_1(h)},\frac{\tilde h}{c_3(\tilde h)}\right)$.
			Thus $j^{(m)}(2x)j^{(m')}\left(-\frac{\log (c_4(\tilde h)x)}{\log c_5(\tilde h)} \right)=0$ if $m,m'\geq 1$ for sufficiently small $\tilde h$.

			We next claim that
 			\begin{align}
 				\notag	& \frac{\mathrm{d}^m}{\mathrm{d}x^m}\chi_{\tilde h}(x) = 2^mj^{(m)}(2x)\left(1-j\left(-\frac{\log (c_4(\tilde h)x)}{\log c_5(\tilde h)} \right)\right)\\
				&   +j(2x) x^{-m}\sum^m_{\a=1} C_{m,\a,\tilde h}  j^{(\a)}\left(-\frac{\log (c_4(\tilde h)x)}{\log c_5(\tilde h)} \right)
			\end{align}
			where $C_{m,\a,\tilde h}=\mathcal{O}\left(\left(-\frac{1}{\log c_5(\tilde h)}\right)^{-m}\right)$ is only depending on $m,\a,\tilde h$ and uniformly bounded in $\tilde h$. We prove this claim by induction in $m$.

			When $m=1$, we immediately obtain the followings by direct calculation:
			\begin{align*}
				& \frac{\mathrm{d}}{\mathrm{d}x}\chi_{\tilde h}(x)\\
				& = 2j'(2x)\left(1-j\left(-\frac{\log (c_4(\tilde h)x)}{\log c_5(\tilde h)} \right)\right)   -\frac{1}{\log c_5(\tilde h)} j(2x)x^{-1} j'\left(-\frac{\log (c_4(\tilde h)x)}{\log c_5(\tilde h)} \right),
 			\end{align*}
			which proves the claim when $m=1$.

 			We next show (3.1) holds for $m=k+1$ if (3.2) holds for $m=k$. We obtain
 			\begin{align*}
 				& \frac{\mathrm{d}^{k+1}}{\mathrm{d}x^{k+1}}\chi_{\tilde h}(x)\\
				& = \frac{\mathrm{d}}{\mathrm{d}x}\left\{\frac{\mathrm{d}^{k}}{\mathrm{d}x^{k}}\chi_{\tilde h}(x)\right\}\\
				& = 2^{k+1}j^{(k+1)}(2x)\left(1-j\left(-\frac{\log (c_4(\tilde h)x)}{\log c_5(\tilde h)} \right)\right)\\
				& +j(2x) \sum^{k+1}_{\a=1} C_{k+1, \a, \tilde h}x^{-k-1} j^{(\a)}\left(-\frac{\log (c_4(\tilde h)x)}{\log c_5(\tilde h)} \right)
			\end{align*}
			where
			\begin{align*}
				& C_{k+1, \a, \tilde h}= \left\{ \begin{array}{l}
				-k C_{k,1,\tilde h} \;\:\qquad \qquad \qquad \qquad \qquad \a=1\\
				-k C_{k,\a,\tilde h} -\frac{1}{\log c_5(\tilde h)}  C_{k,\a-1,\tilde h} \quad 1<\a<k+1\\
				-\frac{1}{\log c_5(\tilde h)}  C_{k,k,\tilde h} \qquad \qquad \qquad \qquad \a=k+1,
	 			\end{array}\right.
 			\end{align*}
			which concludes the proof of the claim.

			Recall $j^{(m)}\left(-\frac{\log (c_4(\tilde h)x)}{\log c_5(\tilde h)} \right)\neq 0$ implies $x\in\left(\frac{\tilde h}{c_1(h)},\frac{\tilde h}{c_3(\tilde h)}\right)$ for any $m>1$.
			If $x\in\left(\frac{\tilde h}{c_1(h)},\frac{\tilde h}{c_3(\tilde h)}\right)$, $x^{-1}<\frac{c_1(h)}{\tilde h}<1$.
			$\left(-\frac{1}{\log c_5(\tilde h)}\right)^{-m}<1$ for sufficiently small $\tilde h$.
			These imply $\frac{1}{\log c_5(\tilde h)}j(2x)x^{-1} j'\left(-\frac{\log (c_4(\tilde h)x)}{\log c_5(\tilde h)} \right)$ is uniformly bounded in $\tilde h$ and $x\in\re$ from Lemma 3.4. Thus $\chi_{\tilde h}$ is an admissible cut-off function.
	 \end{proof}






	\begin{proof}[Proof of Theorem 1.4]

	 Since $1-j_{c_1(h)}(r) + j_{c_1(h)}(r)- j_{\tilde h}(r)=1 - j_{\tilde h}(r)$, it suffices to show $j_{c_1(h)}(r)- j_{\tilde h}(r)u_{c_2^{-1}(\tilde h)}\to0$ as $\tilde h\to 0$.

	Let $F_{\tilde h}(x)=\frac{c_3(\tilde h)}{\tilde h} x \chi_{\tilde h}(x)$. Since $\frac{c_3(\tilde h)}{\tilde h} x$ and all its derivatives are uniformly bounded on supp$\chi_{\tilde h}$ from Lemma 3.4, $F_{\tilde h}$ is an admissible cut-off function.

	We calculate the commutator of $\OP_{F_{\tilde h}}(a)$ and $P-E$ to obtain
	\begin{align*}
	& \jap{v_{\tilde h},[\mathrm{Op}_{F_{\tilde h}}(a),P-E]v_{\tilde h}}_{L^2(\re^n)}\\
	& = \frac{c_3(\tilde h) }{i}\jap{v_{\tilde h},\{r\chi_{\tilde h}(r)a(\rho,\th,\frac{\y}{r}),\rho^2+q( \th,\frac{\y}{r})+V_{\infty}(\th) \}^{\mathrm{w}}(\tilde hx,D_x)v_{\tilde h}}_{L^2(\re^n)}\\
	&  +\mathcal{O}(\tilde h^3)
	\end{align*}
	for any $a\in C^{\infty}_0(\re\times T^{*}S^{n-1})$.

	We also see
	\begin{align*}
	& \{r\chi_{\tilde h}(r)a(\rho,\th,\frac{\y}{r}),\rho^2+q( \th,\frac{\y}{r})+V_{\infty}(\th) \}\\
	& = 2\rho\chi_{\tilde h}(r)a(\rho,\th,\frac{\y}{r})\\
	& + \chi_{\tilde h}(r)\{(\pa_{\rho}a)(\rho,\th,\frac{\y}{r})q( \th,\frac{\y}{r})+(\pa_{\th}a)(\rho,\th,\frac{\y}{r})(\pa_{\y}q)( \th,\frac{\y}{r})\\
	& -(\pa_{\y}a)(\rho,\th,\frac{\y}{r})((\pa_{\th}q)( \th,\frac{\y}{r})+(\pa_{\th}V_{\infty})(\th)+2\rho\frac{\y}{r})\}\\
	& +\left\{4j'(2r)r\left(1-j\left(-\frac{\log (c_4(\tilde h)r)}{\log c_5(\tilde h)} \right)\right)   -\frac{2}{\log c_5(\tilde h)} j(2r) j'\left(-\frac{\log (c_4(\tilde h)r)}{\log c_5(\tilde h)} \right)\right\}\\
	& \times \rho a(\rho,\th,\frac{\y}{r}).
	\end{align*}

	Let $b_{\tilde h}(r,\rho,\th,\y)=4j'(2r)r\left(1-j\left(-\frac{\log (c_4(\tilde h)r)}{\log c_5(\tilde h)} \right)\right)\rho a(\rho,\th,\frac{\y}{r})$.
	For sufficiently small $r$, $b_{\tilde h}(r,\rho,\th,\y)=0$  which implies $b_{\tilde h}$ can be naturall regarded as a function on $T^{*}\re^n$. Also, one can show that $b_{\tilde h}$ is in the symbol class $S$ by direct calculation. Since $(1-j_{\tilde h})v_{\tilde h}\to0$ as $h\to0$, we obtain
	\begin{align*}
		& \lim_{\tilde h\to0}b_{\tilde h}^{\mathrm{w}}(hx,D_x)v_{\tilde h}=\lim_{\tilde h\to0}b_{\tilde h}^{\mathrm{w}}(hx,D_x)j_{\tilde h}v_{\tilde h}=0,
	\end{align*}
	where we have used the fact that $j'(2x)j(x)=0$.

	Let $c_{\tilde h}(r,\rho,\th,\y)= j(2r) j'\left(-\frac{\log (c_4(\tilde h)r)}{\log c_5(\tilde h)} \right)\rho a(\rho,\th,\frac{\y}{r})$.
	For sufficiently small $r$, $c_{\tilde h}(r,\rho,\th,\y)=0$  which implies $c_{\tilde h}$ can be naturall regarded as a function on $T^{*}\re^n$. Also, one can show that $c_{\tilde h}$ is in the symbol class $S$ by direct calculation. Thus we obtain, from Theorem 2.2,
	\begin{align*}
		& \norm{\frac{2}{\log c_5(\tilde h)} c_{\tilde h}^{\mathrm{w}}(\tilde h x, D_x)v_{\tilde h}}_{L^2(\re^n)} \leq -C\frac{2}{\log c_5(\tilde h)}
	\end{align*}
	Thus $\frac{2}{\log c_5(\tilde h)} c_{\tilde h}^{\mathrm{w}}(\tilde h x, D_x)v_{\tilde h}\to0$ as $\tilde h\to0$.

	From Theorem 2.2,
	\begin{align*}
		& \abs{ \jap{v_{\tilde h},[\mathrm{Op}_{F_{\tilde h}}(a),P-E]v_{\tilde h}}_{L^2(\re^n)} }\\
		& \leq \norm{\mathrm{Op}_{F_{\tilde h}}(a)v_{\tilde h}}_{L^2(\re^n)}\norm{(P-E)v_{\tilde h}}_{L^2(\re^n)}\leq C\norm{R_h}_{L^2(\re^n)}.
	\end{align*}
	From the definiton of $R_h$ and $c_3(\tilde h)$, $\norm{R_h}_{L^2(\re^n)}=o(c_3(h))$. Thus we obtains
	\begin{align*}
		& \jap{v_{\tilde h},[\mathrm{Op}_{F_{\tilde h}}(a),P-E]v_{\tilde h}}_{L^2(\re^n)} =o(c_3(\tilde h)).
	\end{align*}

	Let $H$ be a vector field on $T^*(\re\times T^*S^{n-1})$ defined by
	\begin{align*}
	&  H= q( \th,\y)\pa_{\rho}+(\pa_{\y}q)( \th,\y)\pa_{\th} -((\pa_{\th}q)( \th,\y)+(\pa_{\th}V_{\infty})(\th)+2\rho\y)\pa_{\y},
	\end{align*}
	Since $c_3(\tilde h)>c_1(h)^{\frac2{\d}}>\tilde h^2$,
	\begin{align*}
	& o(c_3(\tilde h)) = \frac{c_3(\tilde h)}{i}\jap{v_{\tilde h},\mathrm{Op}_{\chi_{\tilde h}}
	(2\rho a+ Ha)v_{\tilde h}}_{L^2(\re^n)}
	\end{align*}

	We fix a sequence $h_m>0$ such that $h_m\to0$ as $m\to\infty$. By taking subsequence if needed, we may assume $\{v_{h_m}\}$ has defect measure $\m_{\chi}$.
	Then we see
	\begin{align*}
	& \int_{\re\times T^*S^{n-1}} 2\rho a(\rho,\th,\y) + \{(\pa_{\rho}a)(\rho,\th,\y)q( \th,\y)+(\pa_{\th}a)(\rho,\th,\y)(\pa_{\y}q)( \th,\y)\\
	& -(\pa_{\y}a)(\rho,\th,\y)((\pa_{\th}q)( \th,\y)+(\pa_{\th}V_{\infty})(\th)+2\rho\y)\}\mathrm{d}\m_{\chi}=0.
	\end{align*}

	Let $\Phi_t$ be a flow generated by $H$.
	Then we obtain
	\begin{align*}
	&  \frac{\mathrm{d}}{\mathrm{d}t}\int_{\re\times T^*S^{n-1}} \Phi_{t}^{*}(a)e^{2\int^t_0\rho(s)\mathrm{d}s }\mathrm{d}\m_{\chi}=0,
	\end{align*}
	where $\rho(t)$ is defined by $\Phi_{t}(\rho,\th,\y)=(\rho(t),\th(t),\y(t))$.
	Thus
	\begin{align*}
	& \int_{\re\times T^*S^{n-1}} a\mathrm{d}\m_{\chi_c}=\int_{\re\times T^*S^{n-1}} \Phi_{t}^{*}(a)e^{2\int^t_0\rho(s)\mathrm{d}s }\mathrm{d}\m_{\chi}
	\end{align*}
	for any $t\in\re$.

	From Theorem 3.2 and Theorem 2.8, we obtain $\lim_{t\to\infty}\rho(t)\neq 0$ on $\mathrm{supp}\m_{\chi}$ if $E\notin \mathrm{Cv}(V)$.
	If $\rho_{\infty}=\lim_{t\to\infty}\rho(t) >0$, $\int^t_0\rho(s)\mathrm{d}s>\frac12 \rho_{\infty}$.
	 This implies $ \int_{\re\times T^*S^{n-1}} \Phi_{t}^{*}(a)\mathrm{d}e^{2\int^t_0\rho(s)\mathrm{d}s }\m_{\chi}$ diverges as $t\to\infty$ if $\m$ is non-zero.
	 If $\rho_{\infty}=\lim_{t\to\infty}\rho(t) <0$, one can prove $t\to-\infty$ $ \int_{\re\times T^*S^{n-1}} \Phi_{t}^{*}(a)\mathrm{d}e^{2\int^t_0\rho(s)\mathrm{d}s }\m_{\chi}$ diverges as $t\to-\infty$ if $\m$ is non-zero from Lemma 2.5 and Corollary 2.9.
	Thus $\m_{\chi}=0$ and the assertion follows.

	\end{proof}

	\begin{proof}[Proof of Theorem 1.2]
		We fix a sequence $h_m>0$ such that $h_m\to0$ as $m\to\infty$. By taking subsequence if needed, we may assume $\{v_{h_m}\}$ has defect measure $\m_{\chi}$.

		Let $H$ and $\Phi_{t}$ are the same with that in the proof of Theorem 1.4. In the same manner with the proof of the Theorem 1.4,
		\begin{align*}
			\int_{\re\times T^*S^{n-1}} a\mathrm{d}\m_{j}
			=\int_{\re\times T^*S^{n-1}} \Phi_{t}^{*}(a)e^{2\int^t_0\rho(s)\mathrm{d}s }\mathrm{d}\m_{j}
		\end{align*}
		for any $t\in\re$.

		Then we see $\m_j=0$ if $E\notin\mathrm{Cv}(V_{\infty})$, which is contradiction from Theorem 3.1 and the assumption on $u_h$. Thus $E\in\mathrm{Cv}(V_{\infty})$.

		If $E\in \mathrm{Cv}(V_{\infty})$ we see
		\begin{align*}
		\mathrm{supp}(\m_{j})\subset \{(\rho,\th,\y)\in\re\times T^{*}S^{n-1}\mid \lim_{t\to\infty}\rho(t)=0\}.
		\end{align*}
		Let $a\in C^{\infty}_0(\re\times T^{*}S^{n-1};[0,\infty))$ be such that
		\begin{align*}
		\mathrm{supp}(a)\cap \{q(\th,\pa_{\th}V_{\infty}(\th))+q(\th,\y)<\d\}=\phi.
		\end{align*}
		for some $\d>0$.

		Then we see
		\begin{align*}
		& \int_{\re\times T^*S^{n-1}} a\mathrm{d}\m_{j}	\\
		& =\lim_{t\to\infty}\int_{\re\times T^*S^{n-1}} \Phi_{t}^{*}(a) e^{2\int^t_0\rho(s)\mathrm{d}s } \mathrm{d}\m_{j} =\lim_{t\to\infty}\int_{\{\lim_{t\to\infty}\rho(t)=0\}} \Phi_{t}^{*}(a) e^{2\int^t_0\rho(s)\mathrm{d}s } \mathrm{d}\m_{j}.
		\end{align*}
		Since $\rho(t)$ is monotone increasing and $\lim_{t\to\infty}\rho(t)=0$, we obtain $\rho(t)<0$. This implies
		\begin{align*}
		& \lim_{t\to\infty}\int_{\{\lim_{t\to\infty}\rho(t)=0\}} \Phi_{t}^{*}(a)e^{2\int^t_0\rho(s)\mathrm{d}s }\mathrm{d}\m_{j} \leq\lim_{t\to\infty}\int_{\{\lim_{t\to\infty}\rho(t)=0\}} \Phi_{t}^{*}(a)\mathrm{d}\m_{j}.
		\end{align*}
		From Theorem 2.8, the dominant convergence theorem and the fact that $\lim_{t\to\infty}\Phi_{t}^{*}(a)(\rho,\th,\y)=0$ pointwise, we obtain
		\begin{align*}
		& \int_{\re\times T^*S^{n-1}} a\mathrm{d}\m_{j} = \lim_{t\to\infty}\int_{\{\lim_{t\to\infty}\rho(t)=0\}} \Phi_{t}^{*}(a) e^{2\int^t_0\rho(s)\mathrm{d}s }\mathrm{d}\m_{j}\leq 0.
		\end{align*}
		Since $a\geq 0$, we obtain
		\begin{align*}
		& \int_{\re\times T^*S^{n-1}} a\mathrm{d}\m_{j} = 0.
		\end{align*}
		This implies
		\begin{align*}
		\mathrm{supp}(\m_{j})\subset \{(\rho,\th,0)\in\re\times T^{*}S^{n-1}\mid \th\in\mathrm{Cr}(V_{\infty}),\lim_{t\to\infty}\rho(t)=0\}.
		\end{align*}
		If $(\rho,\th,\y)$ is in the set of right-hand side, $\rho(t)=\rho$ for any $t$ since the set of right hand side is contained in the fixed set of $\Phi_t$. Thus we obtain $\rho=0$ and $\th\in V^{-1}_{\infty}(E)$, which concludes the proof.
	\end{proof}

	\section{Example of asymptotic eigenvectors whose defect measure does not vanish}

	In this section, we construct an example of $u_h$ such that the corresponding semiclassical measure is non-zero. We will show the existence of the quasimodes with the following support condition.

	\begin{thm}
		Let $j$ be the same with the function in Section 1 and $E\in\re$, $\th_0\in V_{\infty}^{-1}(E)$ be such that $\pa_{\th}^{k}V_{\infty}(\th_0)=0$ for any $k \leq 2$. Under Assumption A, there exists a solution $u_h$ to the (1.1) which satisfies the following conditions:
		\begin{enumerate}
		\item{$u_h\in\mathcal{D}(P)$ and satisfies
		\begin{eqnarray*}
			\left\{
			\begin{array}{l}
			(P-E)u_h=R_h\\
			\norm{u_h}_{L^2(\re^n)}=1,
			\end{array}
			\right.
		\end{eqnarray*}}
		\item{$\norm{R_h}_{L^2(\re^n)}=o(h)$,}
		\item{$u_h$ satisfies $j(hr)u_h(r,\th)=u_h(r,\th)$,}
		\end{enumerate}

		where $\mathrm{dist}(\cdot,\cdot)$ denotes the distance defined by the metric on $S^{n-1}$ induced by the Euclidean metric on $\re^n$.

	\end{thm}

	\begin{rem*}
		Theorem 4.1 implies semiclassical measure $\m_{j}$ of $u_h$ does not vanish.
	\end{rem*}

	\begin{cor}
		Assume . Let $j$ be the same with the function in Section 1, $R>0$ and $0<\ell<1$ be the same with that in Theorem 1.4, $E\in\re$, and $\th_0\in S^{n-1}$. Under Assumption A, for any $C>0$ there exists a solution $u_h$ to the (1.1) which satisfies the following conditions:
		\begin{enumerate}
			\item{$u_h\in\mathcal{D}(P)$ and satisfies
				\begin{eqnarray*}
				\left\{
				\begin{array}{l}
				(P-E)u_h=R_h\\
				\norm{u_h}_{L^2(\re^n)}=1,
				\end{array}
				\right.
				\end{eqnarray*}}
			\item{$\norm{R_h}_{L^2(\re^n)}=o(1)$ as $h\to0$,}
			\item{$u_h$ satisfies $j(hr)u_h(r,\th)=u_h(r,\th)$,}
			\item{$\mathrm{supp}(u_h)\subset\{(r,\th)\in\re^n\mid r>R,\mathrm{dist}(\th,\th_0)<Cr^{-\ell}\}$ for sufficiently small $h>0$,}
		\end{enumerate}

		where $\mathrm{dist}(\cdot,\cdot)$ denotes the distance defined by the metric on $S^{n-1}$ induced by the Euclidean metric on $\re^n$.
	\end{cor}

	\begin{proof}[Proof of Theorem 4.1]
		(1) We will construct $u_h$ of form $u_h(x)=f_h(r)g_h(\th)$ by the polar coordinate which satisfies following conditions in addition to the conditions in Theorem 4.1:
		\begin{enumerate}
			\item{$\norm{(\pa_r^2+\frac{n-1}r\pa_r)f_h}_{L^2((0,\infty):r^{n-1}\mathrm{d}r)}=o(h)$ as $h\to0$.}
			\item{$\norm{r^{-2}\bigtriangleup_{S^n-1}u_h}_{L^2(\re^n)}=o(h)$, where $\bigtriangleup_{S^n-1}$ denotes Laplacian on $S^{n-1}$.}
			\item{$\abs{V_{\infty}(\th)-E}=o(h)$ on supp$(g_h)$ as $h\to0$.}
		\end{enumerate}

		We may assume that $E=0$ without losing generality since $(V_{\infty}-E)$ is still homogeneous of order zero.

		Let $f\in C^{\infty}_0(1,\infty)\setminus\{0\}$.  We define $f_{h}(r)=C_hh^{-\frac{n}2}f(hr)$, where $C_h>0$ is renormalizing constant. Then we see that $h \leq r^{-1}\leq Ch$ on supp$f_h$ for some $C>0$ and one can easily calculate that $f_h$ satisfies the condition 1. in the beginning of proof.

		Since $\pa_{\th}^{k}V_{\infty}(\th_0)=0$ for any $k \leq 2$, from Taylor's theorem, there exists a small neighbor $U$ of $\th_0$ such that $V_{\infty}(\th)=\mathcal{O}(\mathrm{dist}(\th,\th_0)^{3})$ near $\th=\th_0$.

		Let $\phi \in C^{\infty}_0(\re)$ be such that supp$\phi \subset (-1,1)$ and $\phi(x)=1$ if $\abs{x}\leq \frac12$ and $0\leq\phi\leq1$. We define $\tilde g_h$ by
		\begin{align*}
		&\tilde g_h(\th)=\phi\left(\frac{\mathrm{dist}(\th,\th_0)}{h^{\frac38}}\right).
		\end{align*}
		Then we see that there exists $C>0$ such that $\abs{V_{\infty}}\leq Ch^{\frac{9}8}$ on supp$\tilde g_h$ for sufficiently small $h$. Thus $\abs{V_{\infty}}=o(h)$ as $h\to0$.
		Also, we obtain that $\norm{\bigtriangleup_{S^n-1}\tilde g_h(\th)}=o(h^{-\frac34})\norm{\tilde g_h}$. Let $C_h^{(2)}=\norm{\tilde g_h}_{L^2(S^{n-1})}^{-1}\neq0$ and $g_h=C_h^{(2)}\tilde g_h$.

		Since $\bigtriangleup_{S^n-1}g_h(\th)=o_{L^2(S^{n-1})}(h^{-\frac34})$ and $r^{-2}f_h(r)=\mathcal{O}_{L^2((0,\infty):r^{n-1}\mathrm{d}r)}(h^2)$, we see $r^{-2}\bigtriangleup_{S^n-1}u_h=o_{L^2(\re^n)}(h)$. Combining with the conditions of $f_h$ and $g_h$, we see $(P-E)u_h=o_{L^2(\re^n)}(h)$.

		Actually, we can calculate $\norm{u_h}_{L^2(\re^n)}=1$ from the definition of $f_h$ and $g_h$.
		From the definition of $f_h$, it is obvious that $j(h\abs{x})u_{h}(x)=u_{h}(x)$.
	\end{proof}

	\begin{proof}[Proof of Corollary 4.2]

		Let $E=E_1+E_2$ where $V(\th_0)=E_2$.
		Same as Theorem 4.1, we may assume that $E_2=0$ without losing generality.

		We will construct $u_h$ of form $u_h(x)=f_h(r)g_h(\th)$ which satisfies following conditions in addition to the conditions in Corollary 4.2:
		\begin{enumerate}
			\item{$\norm{(\pa_r^2+\frac{n-1}r\pa_r - E_1)f_h}_{L^2((0,\infty):r^{n-1}\mathrm{d}r)}=o(1)$ as $h\to0$.}
			\item{$\norm{r^{-2}\bigtriangleup_{S^{n-1}}u_h}_{L^2(\re^n)}=o(1)$, where $\bigtriangleup_{S^{n-1}}$ denotes Laplacian on $S^{n-1}$.}
			\item{$V_{\infty}(\th)=o(1)$ on supp$(g_h)$ as $h\to0$.}
		\end{enumerate}

		Let $\phi \in C^{\infty}_0(\re)$ be such that supp$\phi \subset (-1,1)$ and $\phi(x)=1$ if $\abs{x}\leq \frac12$ and $0\leq\phi\leq1$. We define $\tilde g_h$ by
		\begin{align*}
		&\tilde g_h(\th)=\phi\left(\frac{\mathrm{dist}(\th,\th_0)}{h^{\ell}}\right)
		\end{align*}
		$g_h=\tilde C_h\tilde g_h$ where $\tilde C_h$ is a renormalizing constant. Then $V_{\infty}(\th)=o(1)$ on supp$(g_h)$ as $h\to0$.

		Let $f\in C^{\infty}_0(1,\infty)\setminus\{0\}$.  We define $f_{h}(r)=C_hh^{-\frac{n}2}e^{-i\sqrt{E}r}f(hr)$, where $C_h>0$ is renormalizing constant.
		Then we see that
		\begin{align*}
			& \norm{(\pa_r^2+\frac{n-1}r\pa_r - E_1)f_h}_{L^2((0,\infty):r^{n-1}\mathrm{d}r)}=o(1) \text{ and}\\
			& \norm{r^{-2}\bigtriangleup_{S^{n-1}}u_h}_{L^2(\re^n)}=o(1).
		\end{align*}

		From the definition of $f_h$, there exists $C>0$ such that $h^{-1} < r < Ch^{-1}$ on supp$f_h$.
		On the other hand, $\mathrm{dist}(\th,\th_0)<Ch^{\ell}$ on supp$g_h$ for sufficiently small $h$. Thus we see $u_h(r,\th)=f_h(r)g_h(\th)$ satisfies
		\begin{align*}
			& \mathrm{supp}(u_h)\subset\{(r,\th)\in\re^n\mid r>R,\mathrm{dist}(\th,\th_0)<Cr^{-\ell}\}
		\end{align*}
		for sufficiently small $h>0$.

		Actually, we can calculate $\norm{u_h}_{L^2(\re^n)}=1$ from the definition of $f_h$ and $g_h$.

		From the definition of $f_h$, we obtain $j(h\abs{x})u_{h}(x)=u_{h}(x)$, which concludes the assertion.
	\end{proof}

	\section{Proof of Theorem 1.5}
		In this section, we prove observability results for Schr\"odinger operators with homogeneous potentials of order zero.

		\begin{proof}[Proof of Theorem 1.4]
			We prove theorem by constructing sequence of functions $u_m$ such that $ \int^T_0\int_{\O}\abs{e^{-itP}u_m(x)}^2\mathrm{d}x\mathrm{d}t\to0$ as $m\to\infty$.

			Let $X=\{(r,\th)\in\re^n\mid r>R,\mathrm{dist}(\th,\th_0)<Cr^{-\ell}\}$ and $u_h$ be solution of (1.1) which  constructed in Corollary 4.2. Then we can find $\tilde \chi \in C^{\infty}_0(0,\infty)$ such that $\tilde\chi(hr)f_h(r)=f_h(r)$.

			From the assumption of $k$ and $R$, we can take $\f_h \in C^{\infty}(S^{n-1};[0,1])$ so that supp$[\f]\cap \{\th\in S^{n-1}\mid \mathrm{dist}(\th,\th_0)<r^{-\ell}\}=\phi$ and $\tilde\chi(hr)\f_h(\th)=1$ on $\O$ for sufficiently small $h>0$. Then we see that supp$\left[\tilde\chi(hr)\f(\th)\right] \cap X=\phi$ for sufficiently small $h>0$.

			By the assumption on $u_{h_m}$ and $\f$, we see that
			\begin{align*}
			&0\leq\norm{u_{h_m}}_{L^2(\O)}\leq\jap{u_{h_m},\chi_{\O}u_{h_m}}\leq\jap{u_{h_m},\tilde\chi(h_mr)\f_h(\th)u_{h_m}},
			\end{align*}
			where $\chi_{\O}(x)$ denotes characteristic function of $\O$.
			Then from Theorem 4.1 (2) and (3), $j(2h_mr)\f_h(\th)u_{h_m}=0$ for sufficiently large $m$, which means $\norm{u_{h_m}}_{L^2(\O)}=0$ for sufficiently large $m$.

			Next we claim $F_m(t)=\jap{e^{-itP}u_{h_m},\chi_{\O}e^{-itP}u_{h_m}}_{L^2(\re^n)}\to0$ as $m\to\infty$.

			One can calculate as follows:
			\begin{align*}
			& \frac{\mathrm{d}F_m}{\mathrm{d}t}(t)=-i\jap{e^{-itP}Pu_{h_m},\chi_{\O}e^{-itP}u_{h_m}}_{L^2(\re^n)}\\
			& \quad \quad \quad \quad +i\jap{e^{-itP}u_{h_m},\chi_{\O}e^{-itP}Pu_{h_m}}_{L^2(\re^n)}.
			\end{align*}
			Thus we see
			\begin{align*}
			& \Bigabs{\frac{\mathrm{d}F_m}{\mathrm{d}t}(t)}\leq C\norm{u_{h_m}}_{L^2(\re^n)}\norm{(P-E)u_{h_m}}_{L^2(\re^n)}=C\norm{(P-E)u_{h_m}}_{L^2(\re^n)},
			\end{align*}
			where $C>0$ is a constant independent of $t$ and we have used boundedness of $\chi_{\O}$ in the first inequality and uniform boundedness of $u_m$ in the second inequality.

			Since $F_m(t)=F_m(0)+\int^t_0\frac{\mathrm{d}F_m}{\mathrm{d}t}(s)\mathrm{d}s$, we see for $t\in[0,T]$,
			\begin{align*}
			& \abs{F_m(t)}\leq \abs{F_m(0)} +\int^t_0\abs{\frac{\mathrm{d}F_m}{\mathrm{d}t}(s)}\mathrm{d}s \leq \abs{F_m(0)}+ \tilde C\norm{Pu_{h_m}}_{\mathcal{H}}T.
			\end{align*}
			Letting $m\to\infty$, we obtain the claim.


			For any $\e'>0$, there exists sufficiently large $M>0$ so that $m>M$ implies $\abs{\jap{e^{-itP}u_{h_m},\chi_{\O}e^{-itP}u_{h_m}}}_{L^2(\re^n)} \leq \frac{\e'}{T}$. Then $\int^T_0\int_{\O}\abs{e^{-itP}u_m(x)}^2\mathrm{d}x\mathrm{d}t\leq\e'$ for $m>M$, which concludes the proof.
		\end{proof}
	\appendix

\end{document}